\documentclass[12pt, a4paper]{amsart}
\usepackage{amsfonts,latexsym}
\usepackage{amssymb}

\newtheorem{theorem}{Theorem}[section]
\newtheorem{lemma}[theorem]{Lemma}
\newtheorem{corollary}[theorem]{Corollary}

\theoremstyle{definition}
\newtheorem{definition}[theorem]{Definition}

\theoremstyle{remark}
\newtheorem{remark}[theorem]{Remark}
\newcommand{\eps}{\varepsilon}
\newcommand{\N}{{\mathbb N}}
\newcommand{\F}{\mathcal{F}}

\newcommand{\ifff}{if and only if }

\newcommand{\Pe}{\mathcal{P}}

\newcommand{\interior}{\operatorname{int}}

\numberwithin{equation}{section}

\begin{document}
\title [The Schur $\ell_1$ theorem for filters]
{The Schur $\ell_1$ theorem for filters}

\author[A.\,Aviles, B.\,Cascales, V.\,Kadets and A.\,Leonov]
{Antonio Aviles Lopez, Bernardo Cascales Salinas, Vladimir Kadets
and Alexander Leonov}
\thanks{The work of the third author was supported by the Seneca Foundation,
Murcia. Grant no. 02122/IV2/05}
\begin{abstract}
    We study classes of filters $\F$ on $\N$ such that weak and
    strong $\F$-convergence of sequences in $\ell_1$ coincide. We
    study also analogue of $\ell_1$ weak sequential completeness
    theorem for filter convergence.
\end{abstract}

\address {Equipe de Logique Math\'ematique
UFR de Math\'ematiques (case 7012) Universit\'e Denis-Diderot
Paris 7 2 place Jussieu 75251 Paris Cedex 05 France}
\email{avileslo@um.es}
\address{Departamento de Matemáticas, Universidad de Murcia,
30100 Espinardo, Murcia, Spain} \email{beca@um.es}
\address{Department of Mechanics and Mathematics,
Kharkov National University,
 pl.~Svobody~4,  61077~Kharkov, Ukraine}
\email{vova1kadets@yahoo.com}
\address{Department of Mechanics and Mathematics,
Kharkov National University,
 pl.~Svobody~4,  61077~Kharkov, Ukraine}
\email{leonov\_family@mail.ru}

\date{January 12, 2007.}
\maketitle

\section{Preliminaries}

Every theorem of Classical Analysis, Functional Analysis or of the
Measure Theory that states a property of sequences leads to a
class of filters for which this theorem is valid. Sometimes such
class of filters is trivial (say, all filters or the filters with
countable base), but in several cases this approach leads to a new
class of filters, and the characterization of this class can be a
very non-trivial task. Among such non-trivial examples there are
Lebesgue filters (for which the Lebesgue dominated convergence
theorem is valid), Egorov filters which correspond to the Egorov
theorem on almost uniform convergence \cite{kadleon}, and those
filters $\F$ for which every weakly $\F$ convergent sequence has a
norm-bounded subsequence  \cite{gakad}.

One of the reasons to study such questions is that they bring a
new light to the classical results. Say, it is known, that the
dominated convergence theorem can be deduced from the Egorov
theorem. The question, whether the converse is true has no sense
in the classical context: if both the theorems are true, how one
can see that one of them is not deducible from the other one? But
if one looks at the correspondent classes of filters, the problem
makes sense and in fact there are Lebesgue filters which are not
Egorov ones (in particular the statistical convergence filter).

In this paper we study the Schur theorem on coincidence of weak
and strong convergence in $\ell_1$ in a general setting when the
ordinary convergence of sequences is substituted by a filter
convergence. We show that for some filters this theorem is valid
and for some is not and give necessary conditions and sufficient
conditions (close one to another) for its validity. After that we
consider the Schur theorem for ultrafilters. We also study a
related problem of weak sequential completeness for filter
convergence.

Recall that a {\it filter} $\F$ on a set $N$ is a not-empty
collection of subsets of $N$ satisfying the following axioms:
$\emptyset \notin \F$; if $A,B \in \F$ then $A \bigcap B \in \F$;
and for every $A \in \F$ if $B \supset A$ then $B \in \F$. All
over the paper if the contrary is not stated directly we consider
filters on a countable set $N$. Sometimes for simplicity we put
$N=\N$.

A sequence $(x_n)$, $n \in N$ in a topological space $X$ is said
to be $\F$-{\it convergent} to $x$ (and we write $x=\F\text{-}\lim
x_n$ or $x_n \to_\F x$) if for every neighborhood $U$ of $x$ the
set $\{n \in N: x_n \in U\}$ belongs to $\F$.

In particular if one takes as $\F$ the filter of those sets whose
complement is finite (the {\it Fr\'echet filter}), then
$\F$-convergence coincides with the ordinary one.

The natural ordering on the set of filters on $N$ is defined as
follows: $\F_1 \succ \F_2$ if $\F_1 \supset \F_2$. If $G$ is a
centered collection of subsets (i.e. all finite intersections of
the elements of $G$ are non-empty), then there is a filter
containing all the elements of $G$. The smallest filter,
containing all the elements of $G$ is called {\it the filter
generated by} $G$.

Let $\F$ be a filter. A collection of subsets $G \subset \F$
is called  {\it the base of} $\F$ if for every $A \in \F$
there is a $B \in G$ such that $B \subset A$.

A filter $\F$ on $N$ is said to be {\it free} if it dominates the
Fr\'echet filter. All the filters below are supposed to be free.
In particular every ordinary convergent sequence will be
automatically $\F$-convergent.

A maximal in the natural ordering filter is called an {\it
ultrafilter}. The Zorn lemma implies that every filter is
dominated by an ultrafilter. A filter $\F$ on $N$ is an
ultrafilter \ifff for every $A \subset N$ either $A$ or $N
\setminus A$ belongs to $\F$.

A subset of $N$ is called {\it stationary} with respect to a
filter $\F$ (or just $\F$-stationary) if it has nonempty
intersection with each member of the filter. Denote the collection
of all $\F$-stationary sets by $\F^*$. For an $I \in \F^*$ we call
the collection of sets $\{A\cap I:\ A\in\F\}$ {\it the trace of
$\F$ on $I$} (which is evidently a filter on $I$), and by $\F(I)$
we denote the filter on $N$ generated by the trace of $\F$ on $I$.
Clearly $\F(I)$ dominates $\F$. Any subset of $N$ is either a
member of $\F$ or the complement of a member of $\F$ or the set
and its complement are both $\F$-stationary sets. $\F^*$ is
precisely the union of all ultrafilters dominating $\F$. $\F^*$ is
a filter base \ifff it is equal to $\F$ and $\F$ is an
ultrafilter.

\begin{theorem}\label{stationary-thm}
Let $X$ be topological space, $x_n, x \in X$ and let $\mathcal{F}$
be a filter on $N$. Then the following conditions are equivalent

\begin{enumerate}
\item \label{pro1}$(x_n)$ is $\F$-convergent to $x$;
\item\label{pro2}$(x_n)$ is $\F(I)$-convergent to $x$ for every $I
\in \F^*$;
\item \label{pro3}$x$ is a cluster point of $(x_n)_{n
\in I}$ for every $I \in \F^*$.
\end{enumerate}
\end{theorem}
\begin{proof} Implications (\ref{pro1})$\Rightarrow$ (\ref{pro2})
and (\ref{pro2})$\Rightarrow$ (\ref{pro3}) are evident. Let us
prove that (\ref{pro3})$\Rightarrow$ (\ref{pro1}). Suppose $x_n$
do not $\F$-converge to $x$. Then there is  such a neighborhood
$U$ of $x$  that in each $A\in\F$ there is a $j \in A$ such that
$x_j \not\in U$. Consequently $I=\{j \in \N: x_j \not\in U\}$ is
stationary and $x$ is not a cluster point of $(x_n)_{n \in I}$.
\end{proof}

More about filters, ultrafilters and their applications one can
find in most of advanced General Topology textbooks, for example
in \cite{tod}.

For the standard Banach space terminology we refer to \cite{lts}.
All the spaces, functionals and operators (although this does not
matter) are assumed to be over the field of reals. Before we pass
to the main results let us recall some notations and geometric
properties of $\ell_1$. Denote by $e_n$ the $n$-th element of the
canonical basis of $\ell_1$ and by $e_n^*$ the $n$-th coordinate
functional on $\ell_1$. In this notations for every $x \in \ell_1$
we have
$$
x = \sum_{n \in \N} e_n^*(x)e_n.
$$
Recall that $e_n$ are separated from 0 by the functional $f(x) =
\sum_{n \in \N} e_n^*(x)$, i.e. 0 is not a weak cluster point of
$(e_n)$. The following lemma can be easily extracted from the
block-basis selection method (see \cite{lts}, volume 1). We give
the proof for completeness.

\begin{lemma} \label{lem01}
Let $y_n \in \ell_1$, $\inf_{n \in \N}\|y_n\|=\eps_0 > \eps > 0$
and let $\{m(n)\}$ be an increasing sequence of naturals. Denote
$z_i=\sum_{k \in (m(i),m(i+1)]}e_k^*(y_i)e_k$. If under these
notations $\sup_{n \in \N}\|y_n - z_n\| < \eps/2$ (i.e. $(y_n)$ is
a small perturbation of the block-basis $(z_n)$) then $(y_n)$ is
equivalent to the sequence $(\|y_n\|e_n)$ and consequently 0 is
not a weak cluster point of $(y_n)$.
\end{lemma}

\begin{proof}
  We must find $c_1, c_2>0$ such that for every collection of
  scalars $a_n$
  $$
  c_1\sum_{n \in \N}|a_n|\|y_n\| \le \left\|\sum_{n \in \N}a_n y_n\right\|
  \le c_2\sum_{n \in \N}|a_n|\|y_n\|.
  $$
  The upper estimate with $c_2=1$ follows immediately from the
  triangular inequality. The lower one holds with
  $c_1=1-\eps_0/\eps$
  $$
  \left\|\sum_{n \in \N}a_n y_n\right\| \ge \sum_{n \in \N}|a_n|\|z_n\| -
  \sum_{n \in \N}|a_n|\|y_n-z_n\| \ge
  $$
  $$ \sum_{n \in \N}|a_n|\|y_n\| - 2\sum_{n \in \N}|a_n|\|y_n-z_n\| \ge
  \left(1-\frac{\eps}{\eps_0}\right)\sum_{n \in \N}|a_n|\|y_n\|.
  $$
\end{proof}

\section{Simplified Schur property for filters}

There are several natural ways to generalize the Schur theorem for
filters instead of sequences. Let us start with the one leading to
a class of filters which we are able to characterize completely in
combinatorial terms.

\begin{definition} \label{def1}
A filter $\F$ on $\N$ is said to be a {\it simple Schur filter}
(or is said to have the {\it simplified Schur property}) if for
every coordinate-wise convergent to 0 sequence $(x_n) \subset
\ell_1$ if $(x_n)$ weakly $\F$-converges to 0, then
$\F\text{-}\lim \|x_n\| = 0$.
\end{definition}

For an infinite set $I \subset \N$ let us call a {\it blocking} of
$I$ a disjoint partition $D=\{D_k\}_{k\in\N}$ of $I$ into
non-empty finite subsets.

\begin{definition} \label{def2}
A filter $\F$ on $\N$ is said to be {\it block-respecting} if for
every $I \in \F^*$ and for every blocking $D$ of $I$ there is a $J
\in \F^*$, $J\subset I$ such that $|J \cap D_k|=1$ for all $k$,
where the ``modulus'' of a set stands for the number of elements
in the set.
\end{definition}

\begin{remark}If in the definition above one writes
\begin{equation} \label{eq-block}
\forall_{k\in\N}  |J \cap D_k| \le 1
\end{equation}
instead of $|J \cap D_k|=1$, one will obtain an equivalent
definition.
\end{remark}
\begin{remark} \label{rem1.1}
If $\F$ is block-respecting, then $\F(J)$ for every $J \in \F^*$
is also block-respecting.
\end{remark}
\begin{lemma}\label{simple-lemma}
Let $\F$ be a block-respecting filter and let $(x_n) \subset
\ell_1$ form a coordinate-wise convergent to 0 sequence, which
does not $\F$-converge to 0 in norm. Then there is a $J \in \F^*,$
such that the sequence $(x_n), n \in J$ is equivalent to $(a_i
e_i)$, where $e_i$ form the canonical basis of $\ell_1$, $a_i \ge
1$.
\end{lemma}
\begin{proof}
Due to the Theorem \ref{stationary-thm} there is an $I \in \F^*$
such that $\inf_{n \in I}\|x_n\| > \eps > 0$. Fix a decreasing
sequence of $\delta_k > 0$, $\sum_{k \in \N}\delta_k \le \eps/8$.
Using the definition of $\ell_1$ let us select an increasing
sequence of naturals $(m(n))$ and  such that for every $n \in \N$
\begin{equation} \label{eq-block1}
\sum_{k \ge m(n)}|e_k^*(x_n)| < \delta_n
\end{equation}
and using the coordinate-wise convergence of $x_n$ to 0 select an
increasing sequence of integers $(n_i)$ such that $n_0=0$, $D_i :=
(n_{i-1}, n_i] \cap I \neq \emptyset$ and for every $i \in \N$ and
$j \ge n_{i+1}$
\begin{equation} \label{eq-block2}
\sum_{k \le m(n_{i})} |e_k^*(x_j)| < \delta_i.
\end{equation}
Taking in account the respect which $\F$ has to the blocks $D_i$
let us select a $J =\{j_1, j_2, \ldots\}\in \F^*,$ $J \subset I$
such that $j_i \in (n_{i-1}, n_i]$ for all $i\in\N$. Since $J\in
\F^*,$ either $J_1 =\{j_1, j_3, j_5 \ldots\}$ or $J_2 =\{j_2, j_4,
j_6 \ldots\}$ is an $\F$-stationary set as well. Let, say, $J_2\in
\F^*$. Let us show that in fact vectors $y_i=x_{j_{2i}}$  are
small perturbations of the block-basis $z_i = \sum_{k \in
(m(n_{2(i-1)}),m(n_{2i})]}e_k^*(y_i)e_k$, which due to the Lemma
\ref{lem01} completes the proof. So:

$$
\|y_i - z_i \| = \sum_{k \le m(n_{2i-2})}|e_k^*(x_{j_{2i}})|
+ \sum_{k > m(n_{2i})}|e_k^*(x_{j_{2i}})|.
$$
Taking into account inequalities (\ref{eq-block1}),
(\ref{eq-block2}) and that $j_{2i} \in (n_{2i-1}, n_{2i}]$, we get
$\|y_i - z_i \| \le 2 \delta_{j_{2i}}$ which implies the condition
of Lemma \ref{lem01}.

\end{proof}
\begin{theorem}\label{simple-thm}
A filter $\F$ on $\N$ has the simplified Schur property \ifff
$\F$ is block-respecting.
\end{theorem}
\begin{proof}
The ``if'' part of the theorem follows immediately from
Lemma~\ref{simple-lemma}. So let us turn to the ``only if'' part.
Assume that $\F$ is not block-respecting, i.e. there is an $I \in
\F^*$ and there is a blocking $D$ of $I$ such that every $J
\subset I$ satisfying (\ref{eq-block}) is not $\F$-stationary. In
other words $\N \setminus J \in \F$ for every $J \subset I$
satisfying (\ref{eq-block}). Since the finite intersection of the
filter elements again belongs to $\F$, we can reformulate the fact
that $\F$ is not block-respecting as follows: there is an $I \in
\F^*$ and such a blocking $D=\{D_k\}_{k\in\N}$ of $I$ that $\N
\setminus J \in \F$ for every $J \subset I$ satisfying
\begin{equation} \label{eq-block!}
\sup_{k\in\N}|J \cap D_k| < \infty.
\end{equation}
Now, using Dvoretzky's almost Euclidean section theorem let us
select an increasing sequence of integers $0=m_0<m_1<m_2< \ldots$
and a sequence of vectors $x_n \in \ell_1$ such that $x_n=0$ when
$n \not\in I$; $x_n \in {\rm Lin}\{e_k\}_{k \in (m_{i-1}, m_i]}$
when $n \in D_i$ and for every collection of scalars $a_n$
\begin{equation} \label{dvor}
\left( \sum_{n\in D_i}|a_n|^2\right)^{1/2}  \le \left\|\sum_{n \in
D_i}a_n x_n\right\|  \le \left( 2 \sum_{n\in D_i}|a_n|^2
\right)^{1/2}.
\end{equation}
This sequence converges coordinate-wise to 0 and is not
$\F$-convergent to 0 in norm, because $\|x_n\| \ge 1$ for every $n
\in I$. Let us prove $x_n$'s weak $\F$-convergence to 0, which
will show that $\F$ does not have the simplified Schur property.
Well, take an $f \in \ell_1^*$ with  $\|f\| = 1$, fix an $\eps >
0$ and consider the set of indexes $A=\{n: f(x_n) < \eps \}$. We
must prove that $A \in \F$. Since the complement of $A$ lies in
$I$, it is sufficient to show that $J=\N \setminus A = \{n: f(x_n)
\ge \eps \}$ satisfies (\ref {eq-block!}). In other words we must
estimate $d_k=|J \cap D_k|$ from above uniformly in $k$. Let us do
this. Consider $y_k= \sum_{n \in J \cap D_k}x_n$. Then $f(y_k)\ge
\eps d_k$ and due to (\ref{dvor}) $\|y_k\|^2 \le 2 d_k$. Hence
$$
\eps d_k \le f(y_k) \le \sqrt{2 d_k}
$$
and $d_k \le 2/\eps^2$.
\end{proof}

\begin{remark}
One can see that in the ``only if" part of the Theorem
\ref{simple-thm} proof the sequence $(x_n)$ is bounded by
$\sqrt{2}$. So, if one restricts Definition \ref{def1} to the
bounded sequences, the class of filters does not change. In fact
this is a little bit surprising because a weakly $\F$-convergent
sequence can converge to infinity in norm \cite{gakad}. If one
analyzes the characterization \cite{gakad} of those ``good"
filters $\F$ for which every weakly $\F$-convergent sequence has a
norm-bounded subsequence, one can see that every simple Schur
filter is ``good". The only obstacle to see this without
refereeing to \cite{gakad} is the coordinate-wise convergence
which appears in Definition \ref{def1}. To see that this obstacle
is not fatal one really needs to go into the proofs of
\cite{gakad}.
\end{remark}

\section{Schur filters}
Let us pass now to the study of the most natural Schur theorem
generalization, which is easier to formulate, but is much more
complicated to characterize in combinatorial terms.
\begin{definition} \label{def2.1}
A filter $\F$ on $N$ is said to be a {\it Schur filter} (or is
said to have the {\it Schur property}) if for every weakly
$\F$-convergent to 0 sequence $(x_n) \subset \ell_1$, $n \in N$
the $\F\text{-}\lim \|x_n\|$ equals 0.
\end{definition}

Evidently, every Schur filter has the simplified Schur property.
By now we don't know if the converse holds true as well.

To simplify the exposition we mostly consider $N=\N$, but the
general case cannot differ from this particular one.

\begin{definition} \label{def2.2}
$\F$ is said to be a {\it diagonal filter} if for every decreasing
sequence $(A_n) \subset \F$ of the filter elements and for every
$I \in \F^*$ there is a $J \in \F^*$, $J \subset I$ such that $|J
\setminus A_n| < \infty$ for all $n \in \N$.
\end{definition}

\begin{lemma}\label{diag-coord}
If a filter $\F$ on $\N$  is diagonal then for every $I \in \F^*$
and for every coordinate-wise $\F$-convergent to 0 sequence $(x_n)
\subset \ell_1$ there is a $J \in \F^*$, $J \subset I$ such that
$x_n$ coordinate-wise converge to 0 along $J$.
\end{lemma}
\begin{proof}
Fix a decreasing sequence of subsets $U_n$, forming a base of
neighborhoods of 0 in the topology of coordinate-wise convergence.
Define $A_n = \{k \in \N: x_k \in U_n\}$. Since $\F$ is diagonal
there is a $J \in \F^*$, $J \subset I$ such that $|J \setminus
A_n| < \infty$ for all $n \in \N$. This is the $J$ we desire.
\end{proof}
\begin{remark} \label{remark}
As one can see from the proof the only property of the
coordinate-wise convergence topology we used is the countable base
of 0 neighborhoods existence. Also one can easily prove the
inverse to the Lemma \ref{diag-coord} result: if $\F$ is not
diagonal, then there is a $I \in \F^*$ and a coordinate-wise
$\F$-convergent to 0 sequence $(x_n) \subset \ell_1$ such that for
every $J \in \F^*$, $J \subset I$ the sequence $(x_n)$ does not
converge coordinate-wise to 0 along $J$.

Let us demonstrate this inverse theorem. By the negation of the
diagonality definition  a decreasing sequence of $A_n \in \F$ and
an $I \in \F^*$ exist such that if  $S \subset I$ satisfies the
condition $|S \setminus A_n| < \infty$ for all $n \in \N$ then $\N
\setminus S \in \F$. Without loss of generality one may assume
that all the $D_n :=A_n \setminus A_{n+1}$ are infinite and
$\bigcup_n D_n = I$. Then every $J \in \F^*$, $J \subset I$ must
satisfy condition
\begin{equation} \label{eq.stand}
\left|\left\{n \in \N:\left|J\cap D_n \right| = \infty
\right\}\right|= \infty.
\end{equation}
For every $n \in I$ denote by $f(n)$ such index that $n \in
D_{f(n)}$. Consider the following sequence $(x_n)$:  for $n \in \N
\setminus I$ put $x_n = 0$, and for $n \in  I$ put $x_n = e_n +
e_{f(n)}$. This sequence is the one we need.
\end{remark}
\begin{theorem}\label{Schur-thm}
If a filter $\F$ on $\N$  is diagonal and is block-respecting,
then $\F$ has the Schur property.
\end{theorem}
\begin{proof}
Let $(x_n) \subset \ell_1$ be weakly $\F$-convergent to 0. Arguing
``ad absurdum''  suppose that there is an $I \in \F^*$ such that
\begin{equation}\label{separated}
\inf_{n \in I}\|x_n\| > \eps > 0.
\end{equation}
Due to Lemma \ref{diag-coord} there is a $J \in \F^*$, $J \subset
I$ such that $x_n$ coordinate-wise converge to 0 along $J$.  Since
$\F(J)$ is block-respecting (Remark \ref{rem1.1}), the condition
(\ref{separated}) contradicts Theorem \ref{simple-thm}.
\end{proof}

It was shown in Theorem \ref{simple-thm} that the block-respect of
$\F$ is a necessary condition in order to be a Schur filter. Our
next goal is to show that the diagonality of $\F$ is not a
necessary condition. To do this define a special filter on $N$.
Let $D= \{D_n\}_{n \in \N}$ be a disjoint partition of $N$ into
infinite subsets. For every sequence $C=\{C_n\}_{n \in \N}$ of
finite subsets $C_n \subset D_n$ and every $m \in \N$ introduce
the set $B_{m,C}=\bigcup_{n=m}^\infty (D_n \setminus C_n)$. The
sets $B_{m,C}$ form a filter base. Denote the corresponding filter
by $\F_D$. One can easily see that $\F_D$ is an example of not
diagonal block-respecting filter. In fact this filter ``almost"
appeared in Remark \ref{remark}. To make the picture clearer, we
may represent $N$ as an infinite matrix $\N \times \N$, with $D_n
= \{(k,n): k \in \N\}$ being its columns.
\begin{definition} \label{defselfrepr}
A filter $\F$ on $N$ is said to be {\it self-reproducing} if for
every $I \in \F^*$ there is a $J \in \F^*$, $J\subset I$ such that
the structure of the trace of $\F$ on $J$ is the same as of the
original filter $\F$, i.e. there is a bijection $s: N \to J$, that
maps $\F$ into its trace on $J$: $A \in \F \iff s(A) \in \F(J)$.
\end{definition}
\begin{theorem} \label{FD}
$\F_{D}$ is a Schur filter, i.e. diagonality is not a necessary
condition for the filter's Schur property.
\end{theorem}
\begin{proof}
First remark, that a subset $J \subset N$ is $\F_{D}$-stationary
if and only if the condition (\ref{eq.stand}) is met. In
particular, for every infinite subset $M \subset \N$ and for every
selection of infinite subsets $A_n \subset D_n$, $n \in M$ the set
$\bigcup_{n \in M} A_n$ is an $\F_{D}$-stationary set. Let us call
such sets of the form $\bigcup_{n \in M} A_n$ ``standard sets''.
Every $\F_{D}$-stationary set contains a standard subset. Remark
also that the structure of the trace of $\F_{D}$ on a standard
subset $J$ is exactly the same as of the original filter $\F_{D}$,
i.e. $\F_{D}$ is self-reproducing.

To prove the theorem assume contrary that there is a sequence
$(x_n) \subset \ell_1$, $n \in N$ that $\F_{D}$-weakly converge to
zero but the norms do not $\F_{D}$-converge to zero. So there is
an $\eps > 0$ and such a standard set $J \subset N$, that $\|x_n\|
\ge \eps$ for all $n \in J$. According to the previous remark, we
may assume without loss of generality that $J=N$, i.e. $\|x_n\|
\ge \eps$ for all $n \in N$. Passing from $x_n$ to $x_n / \|x_n\|$
we may suppose that  $\|x_n\| = 1$ for all $n$. For every fixed $m
\in \N$ select a subsequence of $D_m' \subset D_m$, such that
$x_n, n \in D_m'$ coordinate-wise converge to an element $y_m \in
\ell_1$. Passing to a new standard set of indexes $\bigcup_{m\in
\N} D_m'$ we reduce the situation to the case when $x_n, n \in
D_m$ converge coordinate-wise to $y_m$ for every $m\in \N$.

Notice that due to the weak $\F_{D}$-convergence to zero of the
whole sequence $(x_n), n \in N$, $y_m$ converge coordinate-wise to
zero. In fact, for arbitrary coordinate functional $e_k^*$ and for
every $\eps > 0$ there is a set of the form $B_{m,C}$ such that
$|e_k^*(x_j)| < \eps$ for all $j \in B_{m,C}$. This means that for
$i \in \N$, $i > m$ we have
$$|e_k^*(y_i)| = \lim_{j \in D_i} |e_k^*(x_j)| \le \eps.$$
This means in its turn the desired coordinate-wise convergence to
zero of $(y_m)$.

Introduce a notation: for $n \in N$ denote by $f(n)$ such index
that $n \in D_{f(n)}$. Put $z_n = x_n - y_{f(n)}$. Consider two
cases. The first one: $\|z_n\| \to_{ \F_{D}}0$. In this case
$\|y_m\| \to 1$ as $m \to \infty$, but on the other hand the
condition $y_{f(n)} = x_n - z_n \stackrel{w}{\rightarrow}_{
\F_{D}}0$ implies ordinary weak convergence of $(y_m)$ to 0, which
is impossible according to the Schur theorem. In the remaining
case, there is a standard set on which $\|z_n\|$ are bounded from
below, so we may again without loss of generality assume that
$\|z_n\| > \eps > 0$ for all $n \in N$.

{\bf Claim.} There is such a standard set $J \subset N$ that the
sequence $(z_n)_{n \in J}$ is equivalent to the canonical basis of
$\ell_1$.

{\bf Proof of the claim.} Fix a decreasing sequence of $\delta_k >
0$, $k \in N$, $\sum_{k \in N}\delta_k \le \eps/8$. Using the
definition of $\ell_1$ let us select naturals $m(n)$ such that for
every $n \in N$ the condition
$$
\sum_{k \ge m(n)}|e_k^*(z_n)| < \delta_n
$$
holds true. Take an arbitrary $n_1 \in D_1$. Now using
consequently the coordinate-wise convergence to 0 of sequences
$(z_n)$, $n \in D_m$ for values of $m= 1, 2, \ 1, 2, 3, \ 1, 2, 3,
4,\ \ldots$ select a sequence $(n_i) \subset N$ in such a way that
$n_2 \in D_1$, $n_3 \in D_2$, $n_4 \in D_1$, $n_5 \in D_2$, $n_6
\in D_3$, etc. (like triangle enumeration of a matrix) and for
every $i \in \N$
$$
\sum_{k \le s(i)} |e_k^*(z_{n_{i+1}})| < \frac{\eps}{2^{i+3}},
$$
where $s(i)$ denotes $\max_{k \le i}m(n_k)$. Under this
construction $J = (n_i)_{i \in \N}$ is a standard set, and
$z_{n_{i}}$ is just a small perturbation of the block-basis $w_i =
\sum_{k \in (s(i-1),m(n_{i})]}e_k^*(z_{n_i})e_k$, which due to the
Lemma~\ref{lem01} means that {\bf the claim is proved}.

Now the last step. Once more without loss of generality assume
that $J \subset N$ from the Claim in fact equals $N$, i.e.
$(z_n)_{n \in N}$ are equivalent to the canonical basis of
$\ell_1$. Then for every bounded sequence of scalars $(a_n)_{n \in
N}$  there is a functional $x^* \in \ell_1^*$ such that $x^*(z_n)
= a_n$ for all $n \in N$. Select these $a_n = \pm 1$ in such a way
that for every $i \in \N$
$$|\{n \in D_i: a_n = 1\}|=|\{n \in D_i: a_n = -1\}|= \infty.$$
Then for the corresponding functional $x^*$ we have for every $i \in \N$
$$
\limsup_{n \in D_i}x^*(x_n) - \liminf_{n \in D_i}x^*(x_n)=
\limsup_{n \in D_i}x^*(z_n) - \liminf_{n \in D_i}x^*(z_n) = 2,
$$
which contradicts weak $\F_D$-convergence of $x_n$.
\end{proof}

\section{Category respecting and strongly diagonal filters and ultrafilters}

Let us introduce one more class of filters, which are
block-respecting and diagonal at the same time.

\begin{definition} \label{def2.3}
$\F$ is said to be {\it strongly diagonal} if for every decreasing
sequence $(A_n) \subset \F$ of the filter elements and for every
$I \in \F^*$ there is a $J \in \F^*$, $J \subset I$ such that
\begin{equation} \label{str.diag}
|(J \cap A_n) \setminus A_{n+1}| \le 1 {\rm \,\, for \,\, all \,\,} n \in \N.
\end{equation}
\end{definition}

According to the Theorem \ref{Schur-thm} all strongly diagonal
filters have the Schur property.

\begin{definition}
A filter $\F$ on $\N$ is said to be {\it category respecting} if
for every compact metric space $K$ and for every family of closed
subsets $(F_A)_{A\in \F}$ of $K$ if
  $$
   F_A \subset F_B, \text{ whenever } B\subset A \text{ in } \F,
  $$
and $K = \bigcup_{A\in \F}F_A $ then $\interior(F_B)\not =
\emptyset$ for some $B\in \F$.
\end{definition}
The obvious examples of category respecting filters are those of
countable base.  Moreover, every filter with a base of cardinality
$k < \mathfrak{m}$ is category respecting (see \cite{Frem}, p.~3-4
for the definition of $\mathfrak{m}$ and Theorem 13A, p.~16 for
the corresponding result). But the Martin Axiom means that
$\mathfrak{m}$ equals the cardinality of continuum, so if we
accept the Martin Axiom together with negation of the continuum
hypothesis, we can go to some filters with uncountable base.

The proof of Schur property for $\ell_1$ using the Baire theorem
as presented in \cite[Propostion 5.2]{con} gives a hint that
category respecting filters are related to the Schur property. The
next theorem shows that in fact to be category respecting is a
stronger restriction than to have the Schur property.

\begin{theorem} \label{thm.2cat} If $\F$ is a category respecting
filter on $\N$, then $\F$ is strongly diagonal.
\end{theorem}
\begin{proof}
Assume contrary that $\F$ is not strongly diagonal, i.e. there is
a decreasing sequence $(A_n) \subset \F$ of the filter elements
and there is an $I \in \F^*$ such that for all $J \in \F^*$, $J
\subset I$ the condition (\ref{str.diag}) is not met. Without loss
of generality we may assume that the filter is defined only on $I$
(pass to the trace of $\F$ on $I$), that $\bigcap_{n \in \N} A_n =
\emptyset$ (this intersection is not stationary, so we may just
erase this intersection from $I$)  and that all $D_n := A_n
\setminus A_{n+1}$ are not empty. If one picks up a sequence of
finite subsets
\begin{equation} \label{not-str.diag}
C_n \subset D_n, \,\,\, \sup_{n \in \N} |C_n| < \infty \,\,\text{ then }\,\,
\N \setminus \bigcup_{n \in \N}C_n \in \F.
\end{equation}
Let us introduce the following compact topological spaces
$\widetilde{D}_n$: if $D_n$ is finite then $\widetilde{D}_n = D_n$
with discrete topology; if $D_n$ is infinite then $\widetilde{D}_n
= D_n \bigcup \{\infty_n\}$ -- the one-point compactification of
$D_n$. Recall that $K=\prod_{n \in \N}\widetilde{D}_n$ is compact
in coordinate-wise convergence topology and metrizable. Define a
family of closed sets $(F_A)_{A\in \F}$ in $K$ as follows:
$$
F_A = \{x \in K: \pi_{n}(x) \in \widetilde{D}_n \setminus A \text{
for all } n \in \N \},
$$
where $\pi_{n}: K \to  \widetilde{D}_n$ stands for the $n$-th
coordinate projection. These sets are closed and have empty
interior (the interior could be non-empty only if for a
sufficiently large $m$  $D_n \cap A = \emptyset$ for all $n \ge
m$, which is not the case because $\bigcup_{k \ge m}D_k = A_m \in
\F$). For every $x \in K$ the set $A(x) = \N \setminus \bigcup_{n
\in \N}\{\pi_{n}(x)\}$ is a filter element (due to
(\ref{not-str.diag})) and $x \in F_A$. So the union of all
$(F_A)_{A\in \F}$ equals $K$. Contradiction.
\end{proof}
\begin{corollary} If $\F$ is a category respecting filter on $\N$, then
$\F$ is a Schur filter.
\end{corollary}
\begin{corollary}\label{example def2.3}
Every filter with a countable base is strongly diagonal.
\end{corollary}

\begin{theorem}\label{ultra-schur}
Under the assumption of continuum hypothesis there is a strongly
diagonal ultrafilter.
\end{theorem}
\begin{proof}
Denote by $\Omega$ the set of all countable ordinals. Let us
enumerate as $(I(\alpha),A(\alpha)),\, \alpha \in \Omega$ all the
pairs $(I,A)$, where $I$ is an infinite subset of $\N$, and $A$ is
a decreasing sequence of infinite subsets of $\N$:
$A(\alpha)=(A_n(\alpha))_{n\in\N}$,  $\N \supset
A_1(\alpha)\supset A_2(\alpha) \ldots$. We construct recurrently
an increasing family $\F_\alpha, \alpha < \omega_1$ of filters
with countable base and an increasing family of sets
$\Omega_\alpha\subset\Omega$, as follows: $\F_1$ is the Frech\'et
filter, $\Omega_1= \emptyset$.  If we have an ordinal of the form
$\alpha + 1$ we proceed as follows: we find the smallest $\beta
\in \Omega\setminus\Omega_\alpha$ such that $I(\beta) \in
\F_\alpha^*$ and such that $A(\beta)$ consists of $\F_\alpha$
elements. Applying Corollary \ref{example def2.3}, we find a $J
\in \F_\alpha^*, \, J \subset I = I(\beta)$ such that
(\ref{str.diag}) holds true for $A_n = A_n(\beta)$. Define
$\F_{\alpha + 1}$ as the filter generated by $\F_\alpha$ and $J$,
and put $\Omega_{\alpha + 1}= \Omega_\alpha \cup \{\beta\}$.

If $\alpha$ is a limiting ordinal, put $\F_\alpha= \bigcup_{\beta
< \alpha}\F_\beta$ and $\Omega_\alpha= \bigcup_{\beta <
\alpha}\Omega_\beta$.

Define the filter $\F$ we need as $\F = \bigcup_{\beta <
\omega_1}\F_\beta$. Let us demonstrate that $\F$ is an
ultrafilter.  To do this we must prove that $\F^* \subset \F$. Let
$B \in \F^*$. Then $B \in \F_\alpha^*$ for all $\alpha$. Let
$\beta\in\Omega$ be the smallest ordinal, for which $I(\beta) = B$
and $A(\beta)$ consists of filter $\F$ elements. Then there is an
$\alpha$, for which all $A_n(\beta)$ belong to $\F_\alpha$. If
$\beta\in\Omega_\alpha$ this means that the pair $(I(\beta),
A(\beta))$ has appeared in our recurrent construction, and a
subset $J$ of $B$ (and hence $B$ itself) was added to the filter.
If not, then not later than on the step $\alpha + 1 + \beta$ this
pair $(I(\beta), A(\beta))$ has appeared in our recurrent
construction and a subset $J$ of $B$ was added to the filter. By
the same argument $\F$ is strongly diagonal.
\end{proof}

Notice that the diagonality of an ultrafilter $\F$ is equivalent
to the following well-known property: $\F$ is a ``P-point of
$\beta \N$''. The consistency of P-points non-existence is a
celebrated result of Shelah \cite{Wim}. So, since every strongly
diagonal filter is diagonal some set theoretic assumption is
needed for the last theorem. By the way in the setting of
ultrafilters a property equivalent to ``block-respect'', called
``Q-point of $\beta \N$" was also studied and the non-existence of
Q-points is also known to be consistent \cite{mill}.

To conclude this section let us present an example of a strongly
diagonal filter which is not category respecting. This example
resembles strongly the proof of Theorem \ref{thm.2cat}. Let $D=
\{D_n\}_{n \in \N}$ be a disjoint partition of $N$ into infinite
subsets. For every sequence $C=\{C_n\}_{n \in \N}$ of finite
subsets $C_n \subset D_n$ introduce the set $B_{C}=\bigcup_{n\in
\N}(D_n \setminus C_n)$. The sets $B_{C}$ form a filter base.
Denote the corresponding filter by $\F_d$.  A set $J \subset \N$
is $\F_d$-stationary \ifff there is an $n \in \N$ such that $|J
\cap D_n| = \infty$.  One can easily see that $\F_d$ is strongly
diagonal. To show that it is not category respecting consider the
same system of subsets $(F_A)_{A\in \F}$ of the same compact $K$
as in the proof of the Theorem \ref{thm.2cat}. The only difference
is that now in the definition of $K$ we don't need to consider the
case of finite $D_n$. These sets $F_A$ are closed, they have empty
interior, but their union contains all the $K$, which would be
impossible if $\F_d$ was category respecting.

\section{Weak sequential completeness theorem for filters}

\begin{definition} \label{def3.1}
A filter $\F$ on $\N$ is said to be {\it weak $\ell_1$ complete}
(or in abbreviated form WC1-filter) if for every $\F$-convergent
in the topology $\sigma(\ell_1^{**}, \ell_1^{*})$ bounded sequence
$(x_n) \subset \ell_1$ its weak* $\F$-limit $x \in \ell_1^{**}$ in
fact belongs to $\ell_1$.
\end{definition}

It is known that every Banach space with the Schur property is
weakly sequentially complete. The next theorem  together with the
Theorem \ref{ultra-schur} shows that the picture for filters is
more colorful.

\begin{theorem}\label{ultra-WC1}
An ultrafilter cannot be weak $\ell_1$ complete.
\end{theorem}
\begin{proof}
Let $\F$ be a (free as always) ultrafilter on $\N$. Consider an
arbitrary $f=(f_1,f_2, \ldots) \in \ell_\infty = \ell_1^{*}$. Then
for the canonical basis $(e_n)$ of  $\ell_1$ we have
$$
\lim_\F f(e_n) = \lim_\F f_n,
$$
which shows that the sequence $(e_n)$ weakly* $\F$-converges to
the functional $\lim_\F$ on $\ell_\infty$, which evidently does
not belong to $\ell_1$.
\end{proof}

To show that a WC1-filter may have no Schur property (and even to
be without the simplified Schur property), let us recall some
elements of statistical convergence theory  \cite{Fast},
\cite{Con_Top}.

A sequence $(x_k)$ in a topological space $X$ is {\it
statistically convergent to $x$} if for every neighborhood $U$ of
$x$ the set $\{k: x_k \not\in U\}$ has natural density 0, where
the natural density of a subset $A\subset \N$ is defined to be
$\delta(A):=\lim_n n^{-1}|\{k\leq n: k\in A\}|$.

Denote $\F_s=\{I\subset\N: \delta(\N \setminus I)=0\}$ and remark
that $\F_s$ is a filter. As it is easy to see, $\F_s$-convergence
and statistical convergence coincide, and a set $J$ is
$\F_s$-stationary provided $\delta(J) \neq 0$.

Recall that a scalar sequence $(x_k)$ is said to be strongly
Cesaro-summable if there is a scalar $x$ such that
$$\lim_{n \to \infty}\frac{1}{n}\sum_{j=1}^n|x-x_j|=0.$$
It is known that a bounded scalar sequence is statistically
convergent \ifff it is strongly Cesaro-summable (for a general
version of this criterion see \cite[Theorem 8]{Con_Two}). Let us
apply this fact.
\begin{theorem}\label{stat-WC1}
$\F_s$ is a WC1-filter but does not have the simplified Schur property.
\end{theorem}
\begin{proof}
Consider the blocking of $\N$ into $D_n = (2^n -1, 2^{n+1}]$.
Every set $J \subset \N$ intersecting each of $D_n$ by no more
than one element, has zero natural density and consequently cannot
be $\F_s$-stationary. Hence $\F_s$ is not block-respecting and by
the Theorem \ref{simple-thm} $\F_s$ does not have the simplified
Schur property.

Let us show now the weak $\ell_1$ completeness of $\F_s$. Let
$(x_n) \subset \ell_1$ be a bounded sequence and let weak*
$\F_s$-limit of $(x_n)$ be equal to an $x^{**} \in \ell_1^{**}$.
This means that for every $f\in \ell_1^{*}$
$$
\lim_{n \to \infty}\frac{1}{n}\sum_{j=1}^n|f(x^{**}-x_j)|=0.
$$
Hence the vectors $\frac{1}{n}\sum_{j=1}^n x_j$ weakly* converge
to $x^{**}$ as $n \to \infty$. By the ordinary weak sequential
completeness of $\ell_1$ this means that $x^{**} \in \ell_1$.
\end{proof}

Our next goal is to show that if one avoids ultrafilters in a
reasonable sense, then the same sufficient condition which we have
for the Schur property works for the WC1 as well.

\begin{definition} \label{def3.2}
A filter $\F$ on $\N$ is said to be a {\it paper filter} ({\it
p-filter}) if all the traces of $\F$ on $\F$-stationary subsets
are not ultrafilters.
\end{definition}

\begin{theorem}\label{paper-thm}
If a p-filter $\F$ on $\N$ is diagonal and is block-respecting
then $\F$ is a WC1-filter.
\end{theorem}
\begin{proof}
Let $(x_n) \subset \ell_1$ be a bounded sequence and let
$\F$-limit of $(x_n)$ in the topology $\sigma(\ell_1^{**},
\ell_1^{*})$ be equal to an $x^{**} \in \ell_1^{**} \setminus
\ell_1$. Consider the standard projection $P:\ell_1^{**} \to
\ell_1$, which maps every element of $\ell_1^{**}$ (i.e. a linear
functional on $\ell_\infty$) into its restriction on $c_0$. Denote
$x = P x^{**}$. Without loss of generality we may assume that $x =
0$: otherwise consider $x_n - x$ instead of $x_n$. This assumption
means that $x_n$ coordinate-wise converge to 0 with respect to the
filter $\F$. Due to the Lemma \ref{diag-coord} there is a $I \in
\F^*$, such that $x_n$ coordinate-wise converge to 0 along  $I$.
Since $\F(I)$ is block-respecting (Remark \ref{rem1.1}), we may
apply Lemma \ref{simple-lemma} to get such a $J \in \F^*$ $J
\subset I$, that the sequence $(x_n)$, $n \in J$ is equivalent to
the canonical basis of $\ell_1$ (here we use also the boundedness
of the sequence). Since $\F(J)$ is not an ultrafilter we can
decompose $J$ into two disjoint $\F$-stationary subsets $J_1$ and
$J_2$. Consider a functional $x^*\in \ell_1^{*}$ which takes value
1 on all $x_n$, $n \in J_1$ and is equal to $-1$ on every  $x_n$,
$n \in J_2$. Then
$$
1 = \lim_{\F(J_1)} x^*(x_n) = x^*(x^{**}) = \lim_{\F(J_2)} x^*(x_n) = -1.
$$
This contradiction completes the proof.
\end{proof}

To proceed further let us introduce the sum and the product of
filters.
\begin{definition}\label{def5.1}
Let $\F_1$, $\F_2$ be filters on $N_1$ and $N_2$ respectively.
Define $\F_1+\F_2$ as the filter on $N_1 \cup N_2$ consisting of
those elements $A \subset N_1 \cup N_2$ that $A \cap N_1 \in \F_1$
and $A \cap N_2 \in \F_2$. The filter $\F_1 \times \F_2$ is
defined on $N_1 \times N_2$ with base formed by the sets $A_1
\times A_2$, $A_1 \in \F_1$, $A_2 \in \F_2$.
\end{definition}
\begin{definition} \label{def5.2}
A filter $\F$ on $N$ is said to have the {\it double Schur
property } if $\F \times \F$ is a Schur filter.
\end{definition}
\begin{theorem}\label{double Schur thm}
Every filter $\F$ with the double Schur property is a WC1-filter
and a Schur filter at the same time.
\end{theorem}
\begin{proof}
Consider such a bounded sequence $(x_n) \subset \ell_1$ that
$\F$-limit of $(x_n)$ in the topology $\sigma(\ell_1^{**},
\ell_1^{*})$ is equal to an $x^{**} \in \ell_1^{**}$. Then the
double sequence $(x_n - x_m)$ is weakly $\F \times \F$-convergent
to 0. According to the double Schur property of $\F$ this implies
that $\|x_n - x_m\| \to_{\F \times \F} 0$, i.e. (due to the
completeness of $\ell_1$) there is an element $x \in \ell_1$ such
that $\|x_n - x\| \to_{\F} 0$. Evidently $x^{**} = x \in \ell_1$.
\end{proof}

\section{Domination by Schur and WC1 filters. Open problems}

\begin{definition} \label{defincrease}
A property $\mathcal P$ of filters (or corresponding class of
filters) is said to be {\it quasi-increasing} if for every $\F \in
\mathcal P$ all the filters of the form $\F(J)$ for every $J \in
\F^*$ have the property $\mathcal P$ as well.
\end{definition}

\begin{remark} \label{remshurtrase}
$\F(J)$-convergence to 0 (in arbitrary fixed topology) of a
sequence $(x_n)$ is equivalent to $\F$-convergence to 0 in the
same topology of the sequence $(x_n \chi_J(n))$. Consequently the
properties defined only through convergence to 0 (like Schur or
double Schur properties) are quasi-increasing.
\end{remark}

\begin{definition} \label{defmono}
A property $\mathcal P$ of filters is said to be {\it decreasing}
if for every $\F \in \mathcal P$ all the filters dominated by $\F$
have the property $\mathcal P$ as well.
\end{definition}

Evidently WC1 filters form a decreasing class. So one can improve
the Theorem \ref{double Schur thm} as follows: every filter
dominated by a filter with the double Schur property is a
WC1-filter. This is an improvement, because of the following
proposition:
\begin{theorem} \label{not monotone}
The Schur property, the double Schur property and moreover every
non-trivial quasi-increasing property $\mathcal P$ of filters are
not decreasing.
\end{theorem}
\begin{proof}
Let $\F_1 \in \Pe$, $\F_2 \not\in \Pe$ be filters on $N_1$ and
$N_2$ respectively. Then $\F=\F_1+\F_2$ is a filter on $N_1 \cup
N_2$  which cannot have the property $\Pe$, because $\F(N_2)
\not\in \Pe$. On the other hand $\F(N_1) \in \Pe$ but $\F(N_1)$
dominates $\F$.
\end{proof}

One can introduce a bit weaker but still reasonable version of the
Schur property, which is decreasing:

\begin{definition} \label{defalmostSchur}
A filter $\F$ on $N$ is said to be an {\it almost Schur filter}
(or is said to have the {\it almost Schur property}) if for every
weakly $\F$-convergent to 0 sequence $(x_n) \subset \ell_1$, $n
\in N$ the norms of $x_n$ are not separated from 0 (or in other
words 0 is a cluster point for $\|x_n\|$, $n \in N$).
\end{definition}

Theorem \ref{stationary-thm} easily implies that a filter $\F$ on
$N$ has the Schur property \ifff all the filters $\F(J)$ for every
$J \in \F^*$ are almost Schur filter.

One can also introduce increasing properties:
\begin{definition} \label{defmono+}
A property $\mathcal P$ of filters is said to be {\it increasing}
if for every $\F \in \mathcal P$ all the filters that dominate
$\F$ have the property $\mathcal P$ as well.
\end{definition}

Evidently the negation of an increasing property is a decreasing
one and contra versa.

\vspace{5 mm}

\begin{definition} \label{defmonobase}
Let $\Pe$ be an increasing (decreasing) class of filters. A class
of filters $\Pe_1 \subset \Pe$ is said to be a {\it basis} for
$\Pe$ if $\Pe$ is the smallest increasing (decreasing) class,
containing $\Pe_1$.
\end{definition}

The problem which looks interesting is to construct explicitly a
class of almost Schur filters, which forms a base for the class of
all almost Schur filters. The same question makes sense for the
negation of property to be almost Schur filter. Such a study was
done in \cite{gakad} for the class of those filters $\F$, that
weak $\F$-convergence of a sequence implies existence of a bounded
subsequence.

\end{document}